\documentclass[a4paper,10pt]{amsart}

\usepackage{inputenc}
\usepackage{color}
\usepackage{graphicx}
\usepackage{times,mathptm}
\usepackage{amsfonts,amscd,amssymb,amsmath}
\usepackage{xypic}
\usepackage{latexsym}
\usepackage{tikz}
\usetikzlibrary{arrows,snakes,backgrounds,decorations.markings}
\usetikzlibrary{patterns,arrows,decorations.pathreplacing}
\usepackage{tikz-cd}
\usepackage{verbatim}

\newtheorem{proposition}{Proposition}[section]
\newtheorem{lemma}[proposition]{Lemma}
\newtheorem{corollary}[proposition]{Corollary}
\newtheorem{theorem}[proposition]{Theorem}

\theoremstyle{definition}
\newtheorem{definition}[proposition]{Definition}
\newtheorem{example}[proposition]{Example}

\theoremstyle{remark}
\newtheorem{remark}[proposition]{Remark}

\newcommand{\proplabel}[1]{\label{prop:#1}}
\newcommand{\propref}[1]{Proposition~\ref{prop:#1}}

\newcommand{\lemlabel}[1]{\label{lem:#1}}
\newcommand{\lemref}[1]{Lemma~\ref{lem:#1}}

\newcommand{\thelabel}[1]{\label{the:#1}}
\newcommand{\theref}[1]{Theorem~\ref{the:#1}}

\newcommand{\corlabel}[1]{\label{cor:#1}}
\newcommand{\corref}[1]{Corollary~\ref{cor:#1}}

\newcommand{\deflabel}[1]{\label{def:#1}}

\newcommand{\exalabel}[1]{\label{ex:#1}}
\newcommand{\exaref}[1]{Example~\ref{ex:#1}}

\newcommand{\seclabel}[1]{\label{sec:#1}}
\newcommand{\secref}[1]{Section~\ref{sec:#1}}

\newcommand\Ker{{\rm Ker}}

\newcommand\Ima{{\rm Im}}

\newcommand\ot{\otimes}
\newcommand\mc{\mathcal}

\newcommand\ov{\overline}

\newcommand\el{\rm el}
\newcommand\rk{\rm rk}

\newcommand\End{\rm End}

\newcommand\wt{\widetilde}

\newcommand\Irr{\rm Irr}

\newcommand\Ind{\rm Ind}

\newcommand\mapright[1]{\smash{\mathop{\longrightarrow}\limits^{#1}}}

\newcommand\mapleft[1]{\smash{\mathop{\longleftarrow}\limits^{#1}}}

\newcommand\hookmapright[1]{\smash{\mathop{\hookrightarrow}\limits^{#1}}}

\newcommand{\req}{\mathrel{\rotatebox{90}{$\Leftrightarrow$}}}

\makeindex
\begin{document}

\title{Generalized characters for glider representations of groups}
\author[F. Caenepeel]{Frederik Caenepeel}
\address{Department of Mathematics, University of Antwerp, Antwerp, Belgium}
\email{Frederik.Caenepeel@uantwerpen.be}
\address{Shanghai Center for Mathematical Sciences, Fudan University, Shanghai, China}
\email{frederik\textunderscore{}caenepeel@fudan.edu.cn}
\author[F. Van Oystaeyen]{Fred Van Oystaeyen}
\address{Department of Mathematics, University of Antwerp, Antwerp, Belgium}
\email{Fred.Vanoystaeyen@uantwerpen.be}

\begin{abstract}
Glider representations can be defined for a finite algebra filtration $FKG$ determined by a chain of subgroups $1 \subset G_1 \subset \ldots \subset G_d = G$. In this paper we develop the generalized character theory for such glider representations. We give the generalization of Artin's theorem and define a generalized inproduct. For finite abelian groups $G$ with chain $1 \subset G$, we explicitly calculate the generalized character ring and compute its semisimple quotient. The papers ends with a discussion of the quaternion group as a first non-abelian example.
\end{abstract}

\thanks{The first author was partly Aspirant PhD Fellow of FWO and partly postdoctoral researcher at the Shanghai Center for Mathematical Sciences.\\
Corresponding author: Frederik Caenepeel}
\maketitle

\section{introduction}

This paper aims to continue the study of glider representations appearing in group theory. As such, it is a follow up of \cite{CVo1}, \cite{CVo2} and more concretely, it starts from the definition of a generalized trace map as defined in \cite{CVo2}. In a preliminary section, we recall the definition of a glider representation and the relevant associated notions we will need. In general, glider representations can be defined for any (positively) filtered ring $FR$, but here we specifically work over the finite algebra filtration $FKG$ of the group algebra $KG$ of finite group $G$ given by a chain of subgroups $1 \subset G_1 \subset \ldots \subset G_d = G$. Specifically, the algebra filtration is given by $F_{-n}KG = 0$, $n > 0$, $F_nKG = KG_n$ for $0 \leq n \leq d$ and $F_nKG = KG$ for $n \geq d$, with $K$ an algebraically closed field of characteristic 0. We also recall the definition of a generalized trace map and a characterization of irreducibility in terms of the value of the associated trace map at the unit element.\\

In \cite{CVo2} we argued how to retrieve the classical character table of a $p$-group $G$ from the generalized character table and we indicated how this extended theory can discern between the groups $Q_8$ and $D_8$. In Section 3 of this paper we introduce the generalization of the set of class functions, that is, functions $G \to \mathbb{C}^D$, $D = \frac{(d+1)(d+2)}{2}$, constant on the conjugacy classes intersected with the different groups appearing in the chain. We denote this set by $\mc{A}(\wt{G})$ and we equip it with a $\mathbb{Z}$-module structure. Every glider representation $\Omega \supset M$ yields such a generalized class function $\chi_M$, but different gliders can yield the same generalized class function. For irreducible gliders, adding information on the one dimensional subspace $M_{\el(M)} = Km$ does yield different class functions. Therefore we define the additive subgroup $R(\wt{G})$ as the free $\mathbb{Z}$-module generated by $\chi_M$ with $M$ running over the irreducible glider representations. From \cite{CVo1} we know how to induce glider representations from a normal subgroup $H \triangleleft G$ to the bigger group $G$, which brings us into the classical setting where Artin's theorem can be proved. The section ends with a generalization of this result, see \theref{Artin}.\\

In Section 4 we define a $\mathbb{C}$-bilinear map
$$ \langle -,- \rangle: R(\wt{G}) \times R(\wt{G}) \to \mathbb{C}^D,$$
based on the various inproducts of the representation rings $R(G_i)$. For the chain $1 \subset G$ with $G$ abelian, we give a characterization of this generalized inproduct $\langle \chi_M, \chi_M \rangle$ for irreducible, bodyless gliders of essential length 1. From \cite[Theorem 3.13]{CVo2} we argue why irreducible gliders of essential length 1 are indexed by the subsets of $G$ and we prove that the set $\{ \chi_A~\vline~ A \in \mc{P}(G)\}$ is $\mathbb{C}$-linearly independent, \theref{charindep}. We further show that the ring extension $R(G) \hookmapright{} R(\wt{G})$ is integral, \propref{integralextension}. The section ends with a discussion on a chain $1 < H < G$ of finite abelian groups.\\

Finite abelian groups $G$ are isomorphic to their character groups $\hat{G}$, which allows to consider the algebra extension $\mathbb{Q}(G) \hookmapright{} \mathbb{Q} \ot_\mathbb{Z} R(\wt{G})$. Here the generalized character ring is computed with regard to the chain $1 < G$ and we denote $\mathbb{Q}(\wt{G}) : = \mathbb{Q} \ot_\mathbb{Z} R(\wt{G})$. In \secref{primitiveidempotents} we calculate the Jacobson radical $J$ and the semisimple quotient $\wt{Q}(\wt{G})/J$, \theref{Ja}. Combined with the classification of primitive central idempotents in the group algebra $\mathbb{Q}(G)$, see \cite{JLP}, we classify the primitive central idempotents in $\mathbb{Q}(\wt{G})$, \theref{primitiveidempotents}. \secref{complicatedchain} extends the results from \secref{primitiveidempotents} for a chain $1< H <G$ of finite abelian groups and in the final section we calculate the generalized character ring for the non-abelian group $Q_8$.\\

In forthcoming work we give more general results on the generalized character ring associated to the chain $e \leq G$ for $G$ any finite group. These results will allow to distinct between various isocategorical groups such as $Q_8$ and $D_8$ and the non-abelian groups of prime cube order $p^3$ for any prime $p$. These results will also raise further questions related to the representation theory of finite groups.

\section{Preliminaries}

Consider a chain of groups $1 \subset G_1 \subset \ldots \subset G_d = G$ and associated algebra filtration $F\mathbb{C}G$, i.e. $F_{-n}\mathbb{C}G = 0$ for $n >0$,$F_i\mathbb{C}G = \mathbb{C}G_i$ for $0 \leq i \leq d$, $F_n\mathbb{C}G = \mathbb{C}G$ for $n \geq d$. 
\begin{definition}
An $(F\mathbb{C}G$-)glider representation is a $K$-vector space $M$, embedded in a (left) $\mathbb{C}G$-module $\Omega$, together with a chain of descending chain of subspaces
$$\Omega \supset M \supset M_1 \supset \ldots \supset M_n \supset \ldots$$
such that $\forall i \leq j$ the $\Omega$-action of $\mathbb{C}G_i$ maps $M_j$ into $M_{j-i}$ and it holds that 
$$\mathbb{C}G_i M_j \subset M_{j-i} \cap M_i^*,$$
where $M_i^* = \{ m \in M~\vline~ \mathbb{C}G_im \subset M \}.$
\end{definition}
In \cite{NVo1}, \cite{CVo1}, the authors first define the more general notion of an $FR$-fragment, for $FR$ some positively filtered ring. A glider representation then appears as a particular example.\\

Let $\Omega \supset M \supset \ldots \supset M_n \supset \ldots$ be an $F\mathbb{C}G$-glider representation. The intersection $B(M) = \bigcap_i M_i$ is a $\mathbb{C}G$-module and we call it the body of the glider $\Omega \supset M$. If there exists $e \geq 0$ such that $M_e \supsetneq M_{e+1}$ and $M_{e+1} = B(M)$, we say that the glider has finite essential length $\el(M) = e$. From general fragment theory (see \cite{CVo1},\cite{EVO}) we may reduce to the situation where $B(M) = 0$ and $\el(M) = d$, i.e. $M_d \supsetneq M_{d+1} = 0$. For $j \leq i$ we have $\mathbb{C}G_j M_i \subset M_{i-j}$, which we present in the following upper triangular matrix
\begin{equation}\nonumber
\resizebox{0.9 \hsize}{!}{$A(M) = \left(\begin{array}{cccccccc}
M_0  & M_1 & \ldots & M_{i-1} & M_i & \ldots & M_{d-1} & M_d\\
 & G_1M_1 & \ldots & G_1M_{i-1} & G_1M_i & \ldots & G_1M_{d-1} & G_1M_d\\
 &  & \ddots & \vdots& \vdots & \ddots & \vdots & \vdots\\
 &  &  & G_{i-1}M_{i-1} & G_{i-1}M_i & \ldots & G_{i-1}M_{d-1} & G_{i-1}M_d\\
 &  & &  & G_iM_i & \ldots & G_iM_{d-1} & G_iM_d\\
 & &  & &  & \ddots & \vdots & \vdots\\
& & & & & & G_{d-1}M_{d-1} & G_{d-1}M_d\\
 &  &  &  &  &  &  & G_dM_d
\end{array}\right)$}
\end{equation}
The $i$-th row of $A(M)$ consists of $\mathbb{C}G_i$-modules $G_iM_j$. We denote the associated characters by $\chi_{ij}$. We propose the following definition
\begin{definition}\deflabel{gentrace}
Let $\Omega \supset M$ be an $F\mathbb{C}G$-glider representation of essential length $d$ and $B(M) = 0$, then the associated generalized trace\index{generalized trace} of $\Omega \supset M$ is the map $\chi_M: G \to \mathbb{C}^D$, $D = \frac{(d+1)(d+2)}{2}$ which sends $g \in G_i \setminus G_{i-1}$ to
$$\chi_M(g) = \left(\begin{array}{cccccccc}
0  & \ldots & 0  & 0 & 0 & \ldots& 0 & 0 \\
 & \ddots & \vdots & \vdots & \vdots & \ddots & \vdots & \vdots\\
 & & 0&  0 & 0 & \ldots & 0 &  0\\
&  &  & \chi_{i,i}(g) & \chi_{i,i+1}(g) & \ldots  & \chi_{i,d-1}(g)&  \chi_{i,d}(g)\\
& & & & \chi_{i+1,i+1}(g) & \cdots & \chi_{i+1,d-1}(g)&  \chi_{i+1,d}(g)\\
& & &   & & \ddots & \vdots & \vdots\\
&&&&&& \chi_{d-1,d-1}(g) & \chi_{d-1,d}(g)\\
 && & & && & \chi_{d,d}(g)\end{array} \right)$$
\end{definition}
We have written the image $\chi_M(g)$ in matrix form, but it really lives inside $\mathbb{C}^D$.  Clearly $\chi_M(g) = \chi_M(g')$ if and only if $g \in c(g') = \{ hg'h^{-1}~ \vline ~h \in G\}$ and $g,g' \in G_i \setminus G_{i-1}$. By our notation, we have that $\chi_M(g)_{ij} = \chi^M_{i,j}(g)$. The matrix $\chi_M(1)$ lists the dimensions of the $G_jM_i$ and we have the following nice characterization, see \cite{CVo2}.

\begin{proposition}\proplabel{anti-diagonal}
Let $\Omega \supset M$ be a glider representation with respect to a finite group algebra filtration. Then $M$ is irreducible of essential length equal to the filtration length and with $B(M) = 0$ if and only if the matrix $\chi(1)$ is symmetric with respect to the diagonal and has a 1 in the upper right corner.
\end{proposition}
In loc. cit. we argued how to retrieve the classical character table of a $p$-group $G$ from the generalized character table and we indicated how this extended theory can discern between the groups $Q_8$ and $D_8$.

\section{Artin's theorem}

We denote by $\mc{A}(\wt{G})$ the set of class functions\index{class function}, that is, maps $G \to K^D$ that are constant on $c(g) \cap G_{i} \setminus G_{i-1}$ for $g \in G_i \setminus G_{i-1}$. For any integer $n \in \mathbb{Z}$ we define the class function $c_n$ by the map that sends $g \in G_i \setminus G_{i-1}$ to 
$$c_n(g) = \left(\begin{array}{ccccc}
0 & \ldots & 0 & \ldots & 0\\
& \ddots & \vdots & \ddots & \vdots\\
&& n & \ldots & n\\
&&& \ddots & \vdots \\
&&&& n
\end{array}\right) \mapleft{} i{\rm-th ~row}$$

Let $H \triangleleft G$ be a normal subgroup and consider the situation from \cite{CVo1}, that is, we start with an $FKH$-glider $N$ and induce it to an $FKG$-glider $M = N^G$. To construct the induced glider we choose a set map $\sigma: G/H \to G$ such that $\pi \circ \sigma = id_{G/H}$. By choosing a sequence of transversal sets $T_1 \subset T_2 \subset \ldots \subset T_d = T$ ($T_i$ transversal for $G_i/H_i$) one obtains a 2-cocycle $h: G/H \times G/H \to H$ by the equality $\sigma(\ov{g})\sigma(\ov{g'}) = \sigma(\ov{gg'})h(\ov{g},\ov{g'})$. To define the induced glider we further need that for all $g \in G$, the map $h(-,\ov{g}): G/H \to H$ is is restricting to $G_i/H_i \to H_i$ for all $i$. We say that the 2-cocycle is filtered. We define $N^G_i = KT \ot N_i$. Under the condition that all $G_i$ are normal in $G$ we have that the $H_i$ are normal in $G$, from which it follows that for any $i \leq j$
$$G_i(KT \ot N_j) = KT \ot H_iN_j.$$
As a vector space, $KT \ot H_iN_j$ is the direct sum $\bigoplus_{t \in T} Kt \ot H_iN_j$, so to calculate the trace for $g \in G_i$ we must have that $\sigma(\ov{gt}) = \sigma(\ov{t}) = t$. If $g = t_1h_1$, this is equivalent to
$$\begin{array}{c} 
t = \sigma(\ov{gt}) = \sigma(\ov{g})\sigma(\ov{t})h(\ov{g},\ov{t})^{-1} = t_1th(\ov{g},\ov{t})^{-1}\\
\req \\
t^{-1}t_1t = h(\ov{g},\ov{t}) \in H_i.
\end{array}$$
So we must sum over $t \in T$ such that $t^{-1}t_1t \in H_i$ or, equivalently, such that 
$$t^{-1}gt = t^{-1}t_1h_1t \in H_i.$$
Moreover, in this case we have that 
$$g \cdot t \ot n = t \ot t^{-1}t_1tt^{-1}h_1tn = t \ot t^{-1}gt n.$$
We conclude that 
$$\chi^{\wt{G}}_M(g)_{ij} = \sum_{t \in T, ~ t^{-1}gt \in H_i} \chi^{\wt{H}}_N(t^{-1}gt)_{ij}.$$
Again, since all $G_i$ are normal in $G$, we have that 
$$t^{-1}gt \in H_i \Leftrightarrow (th)^{-1}g(th) \in H_i {\rm~for~all~} h \in H.$$
Hence
\begin{equation}\label{indu}
\chi^{\wt{G}}_M(g)_{ij} = \frac{1}{|H|} \sum_{\begin{array}{c} g_0 \in G, \\ g_0^{-1}gg_0 \in H_i \end{array}}  \chi_N^{\wt{H}}(g_0^{-1}gg_0)_{ij}.
\end{equation}

The set $\mc{A}(\wt{G})$ carries a $\mathbb{Z}$-module structure
$$\mathbb{Z} \times \mc{A}(\wt{G}) \to \mc{A}(\wt{G}),$$
which maps $(n, f)$ to the generalized class function which is the component wise multiplication of $c_n$ and $f$. We want to generalize the notion of the character ring $R(G)$ of a group, which can be defined as the additive subgroup of the set of class functions generated by the irreducible characters. However, some generalized traces associated to glider representations are equal.

\begin{proposition}\proplabel{same}
Let $\Omega \supset M, \Omega' \supset M'$ be irreducible glider representations. The associated trace maps $\chi_M$ and $\chi_{M'}$ are equal if and only if the matrices $A(M)$ and $A(M')$ only differ in the top right corner.
\end{proposition}
\begin{proof}
Follows by the structure of irreducible glider representations over finite algebra filtrations, see \cite[Lemma 2.5]{CVo1} and by the fact that $G_i$-representations are determined by their characters. In fact, if $\chi_M = \chi_{M'}$ then the irreducible gliders only differ in the one dimensional $K$-vector spaces $M_d$ and $M'_d$.
\end{proof}

In \cite{EVO} the authors introduced the notion of a fragment direct sum $M \dot{\oplus} M'$ which is defined by $(M \dot{\oplus} M')_i = M_i + M'_i$ and for some $i \leq \el(M),\el(N)$ the sum $M_i + M'_i$ is direct. By the strong fragment direct sum we mean $(M \oplus M')_i = M_i \oplus M'_i$ for all $i \geq 0$.  For a glider $\Omega \supset M$ we have that $n \cdot \chi_M = \chi_{M^{\oplus n}}$ where $M^{\oplus n}$ is the strong fragment direct sum of $n$ times $M$. Also, every finitely generated glider representation is the fragment direct sum of irreducible ones \cite[Theorem 4.7]{EVO}. We could introduce the additive subgroup $R(\wt{G})$ of $\mc{A}(\wt{G})$ consisting of all linear combinations
$$\sum_{i=1}^l c_{n_i} \chi_{M_i},$$
where $n_i \in \mathbb{Z}$ and $M_i$ are irreducible glider representations. However, the previous proposition shows that various irreducible glider representations can yield the same generalized trace map. Therefore, we add to $\chi_M$ the information given by the one dimensional vector space $M_d = Ka$. By doing so, we do have that non isomorphic irreducible glider representations do yield different trace maps. We use the same notation $R(\wt{G})$ for this additive group. In the classical group case, the ring of class functions $R(G)$ carries a natural multiplication and it holds that the product of characters correspond to the character of the tensor product of the representations. The category $KG$-mod is monoidal, which allows to define the tensor product of fragments and glider representations.

\begin{proposition}
Let $\Omega \supset M$, $\Omega' \supset M'$ be $FKG$-glider representations of essential length $\leq d$. Then the descending chain
$$ \Omega \ot \Omega' \supset M \ot M' \supset M_1 \ot M_1' \supset \ldots \supset M_d \ot M_d' \supset 0 \supset \ldots$$
is an $FKG$-glider representation.
\end{proposition}
\begin{proof}
All tensor products are over $K$, which explains the inclusions. The fragment conditions are satisfied since the comultiplication $\Delta: KG \to KG \ot KG$ is given by $\Delta(g) = g \ot g$, extended linearly.
\end{proof}
\begin{definition}
Let $\Omega \supset M \supset M_1 \supset \ldots$, $\Omega' \supset M' \supset M'_1 \supset \ldots$ be $FKG$-gliders. Their tensor product\index{fragment!tensor product} is the glider with chain
$$\Omega \ot \Omega' \supset M \ot M' \supset M_1 \ot M_1' \supset \ldots.$$
\end{definition}

In general, however, it is not the case that the generalized character $\chi_{M \ot M'}$ associated to the tensor product $M \ot M'$ of two glider representations will correspond to the component wise multiplication of $\chi_M$ and $\chi_{M'}$. Consider for example the chain $1 \subset G$, then the irreducible gliders are determined in \cite[Theorem 3.13]{CVo2}. This result shows that the tensor product of such irreducible gliders is often not irreducible and splits by \cite[Theorem 4.7]{EVO} into the fragment direct sum of one irreducible glider of essential length 1 and the direct sum of some fragments isomorphic to $K \supset 0 \supset \ldots$ We will see some examples below.\\

We do want to incorporate the behavior of the tensor product of gliders in our generalized character theory and therefore we define a multiplication on $R(\wt{G})$ in the following way: let $M, M'$ be two glider representations, then
$$\chi_{M}\chi_{M'} := \chi_{M \ot M'}.$$

The ${}^\sim$ refers to the group together with its chain of subgroups. In what follows it will always be clear which chain of subgroups we are working with. We call $R(\wt{G})$ the generalized character ring of the group $G$ with given chain of subgroups. From formula \eqref{indu} and the results for ordinary group representations, see \cite[chapter 9]{Serre}, we have a map
$$\Ind_{\wt{H}}^{\wt{G}}: R(\wt{H}) \to R(\wt{G}),~ \chi \mapsto \Ind_{\wt{H}}^{\wt{G}}\chi.$$
\begin{proposition}
$\Ind_{\wt{H}}^{\wt{G}}R(\wt{H}) \triangleleft R(\wt{G})$ is an ideal.
\end{proposition}
\begin{proof}
Let $M$ be an $FKH$-glider, $N$ an $FKG$-glider. We may assume that both gliders are irreducible of essential length $d$. Denote $M_d = Km$ and $N_d = Kn$. Let $j \leq i$ and $g \in G_j$. As vector spaces $(M^G \ot N)_i = \oplus_{t \in T} Kt \ot M_i \ot N_i = (M \ot N_H)^G$. Let $t \in T, m \in M_i, n \in N_i$. If $g = t'h$, then
$$g \cdot (t \ot m \ot n) = \ov{gt} \ot h(\ov{g},\ov{t}) t^{-1}h tm \ot gn$$
in $(M^G \ot N)$ and
$$g \cdot (t \ot m \ot n) = \ov{gt} \ot h(\ov{g},\ov{t}) t^{-1}h tm \ot h(\ov{g},\ov{t}) t^{-1}h tn$$
in $(M \ot N_H)^G$. To calculate the character $KG_j(M^G \ot N)_i$ and $KG_j(M \ot N_H)^G_i$ one has to sum over the $t \in T$ such that $t^{-1}gt \in H_i$. For these $t$ we have shown above that $h(\ov{g},\ov{t}) t^{-1}h t = t^{-1}gt$, from which it follows that 
$$\chi_{M^G \ot N}(g)_{ij} = \chi_{(M \ot N_H)^G}(g)_{ij}.$$
By definition, $(M^G \ot N)_d = K (\sum_{t \in T} t \ot m_d \ot n_d)$ and $(M \ot N_H)^G_d = K( \sum_{t \in T} t \ot (m_d \ot n_d))$. \propref{same} then entails that $\chi_{M^G}\chi_{N} = \chi_{M^G \ot N}= \chi_{(M \ot N_H)^G}  \in \Ind_{\wt{H}}^{\wt{G}}R(\wt{H})$.
\end{proof}

Inspired by the classical case, we introduce the following class functions: let $\mc{C}$ denote the set of all cyclic subgroups of $G$ and let $H \in \mc{C}$. Then we obtain a chain of cyclic groups $H_i = H \cap G_i$ and we define for $h \in H_i \setminus H_{i-1}$
$$\chi_{\wt{H}}(h)_{jl} = \left\{ \begin{array}{cl}
|H| & {\rm if~}  i \leq j \leq l {\rm~and~}\langle h\rangle = H_i,\\
0 & {\rm otherwise.}
\end{array}\right.$$
This defines a generalized class function for the induced chain on $KH$ and via formula \eqref{indu} we define the induced character $\Ind_{\wt{H}}^{\wt{G}}\chi_{\wt{H}}$: we calculate for $g \in G_i \setminus G_{i-1}$ and $i \leq j \leq l$
\begin{eqnarray*}
\sum_{H \in \mc{C}} \Ind_{\wt{H}}^{\wt{G}} \chi_{\wt{H}}(g)_{jl} &=& \sum_{H \in \mc{C}} \frac{1}{|H|} \sum_{\begin{array}{c} g_0 \in G, \\ g_0^{-1}gg_0 \in H_i \end{array}}  \chi_{\wt{H}}(g_0^{-1}gg_0)_{jl}\\
&=& \sum_{H \in  \mc{C}} \frac{1}{|H|} \sum_{\begin{array}{c} g_0 \in G, \\ \langle g_0^{-1}gg_0\rangle  = H_i  \end{array}}  |H|\\
&=& |G| = c_{|G|}(g)_{jl}.
\end{eqnarray*}
Hence we have proven
\begin{lemma}\lemlabel{classe}
The following equality holds
$$\sum_{H \in C} \Ind_{\wt{H}}^{\wt{G}} \chi_{\wt{H}} = c_{|G|}.$$
\end{lemma}
The equation of the previous lemma is an equality inside $\mc{A}(\wt{G})$, that is, as generalized class functions. The classical result \cite[Proposition 28]{Serre} which states that the constant class function 
$$\chi_H: H \to K, ~ h \mapsto \left\{ \begin{array}{cl}
|H| & {\rm if~} \langle h \rangle = H,\\
0 & {\rm otherwise} \end{array}\right.$$
belongs to $R(H)$ for a cyclic group $H$ remains valid in our situation:
\begin{lemma}
$\chi_{\wt{H}} \in R(\wt{H})$ for all $H \in  \mc{C}$.
\end{lemma}
\begin{proof}
For $H = e$, we have that $\chi_{\wt{e}} = \chi_{K \supset K \supset \ldots \supset K \supset 0 \supset \ldots}$, whence by \propref{same} we have that $\chi_{\wt{e}} \in R(\wt{e})$. We can now use induction on $|H|$, \lemref{classe} and the fact that $c_{|H|} = c_{|H|}\chi_{T} \in R(\wt{H})$ where 
$$T \supset T \supset \ldots \supset T \supset 0 \supset \ldots$$
is the associated glider character with $T$ the trivial $H$-representation. 
\end{proof}
We have the generalization of Artin's theorem, \cite[Theorem 17]{Serre}
\begin{theorem}\thelabel{Artin}
If $\Omega \supset M$ is an $FKG$-glider representation then $\chi_M$ is a `rational' linear combination of characters induced from $FKH$-glider representations where $H$ runs over the cyclic subgroups of $G$.
\end{theorem}
\begin{proof}
By \lemref{classe} we know that $c_{|G|} \in \sum_{H \in  \mc{C}} \Ind_{\wt{H}}^{\wt{G}} R(\wt{H})$. Since the latter is an ideal, it also contains $c_{|G|}\chi_M$, whence the result.
\end{proof}

\section{ring structure on $R(\wt{G})$}

From now on, we will work over $K = \mathbb{C}$.\\

Classically, the ring of class functions $R(G)$ carries an inproduct $\langle -,- \rangle$ defined by
$$\langle f, g \rangle = \frac{1}{|G|} \sum_{x\in G} f(x)g(x^{-1})^*.$$
It is well-known that the characters $\chi_i$ associated to the irreducible representations $V_i$ - we call them the irreducible characters - form an orthonormal basis for $R(G)$. We would like to generalize this to glider characters for a chain of groups $1 \subset G_1 \subset \ldots \subset G_d = G$. For $f \in R(\wt{G})$ we denote $f = (f_{ij})_{ij}$ with 
$$f(x) = \left( \begin{array}{ccccc}
f_{00}(x) & \cdots & f_{0i}(x) & \cdots & f_{0d}(x)\\
& \ddots & \vdots & \ddots & \vdots\\
&& f_{ii}(x) & \cdots & f_{id}(x)\\
&&& \ddots & \vdots \\
&&&& f_{dd}(x)
\end{array}\right)$$

\begin{definition}
We define a $\mathbb{C}$-bilinear map
$$ \langle -,- \rangle: R(\wt{G}) \times R(\wt{G}) \to \mathbb{C}^D,$$
given by
$$ \langle f, g \rangle = \left( \begin{array}{ccccc}
\langle f_{00},g_{00} \rangle^{G_0} & \ldots & \langle f_{0i},g_{0i} \rangle^{G_0} & \ldots & \langle f_{0d}, g_{0d}\rangle^{G_0}\\
& \ddots & \vdots & \ddots & \vdots\\
&& \langle f_{ii},g_{ii} \rangle^{G_i} & \ldots & \langle f_{id},g_{id}\rangle^{G_i}\\
&&& \ddots & \vdots \\
&&&& \langle f_{dd},g_{dd}\rangle^{G_d}
\end{array}\right)$$
where the upper index $G_i$ indicates for which group the ordinary inproduct is calculated.
\end{definition}

Let us consider the easiest chain $1 \subset G$. In the case when $G$ is abelian, we have the following generalization of the orthonormality relations
\begin{proposition}\proplabel{ortho}
Let $G$ be a finite abelian group. An $FKG$-glider representation $\Omega \supset M  \supset M_1 \supset 0 \supset \ldots$ of essential length 1 and zero body is irreducible if and only if 
$$ \langle \chi_M, \chi_M \rangle = \left( \begin{array}{cc}
n^2 & 1 \\
0 & n \end{array} \right)$$
for some $n \leq |G|$.
\end{proposition}
\begin{proof}
Suppose that $M \supset M_1$ is irreducible. By \cite[Theorem 3.13]{CVo2} we know that $M_1 = Ka$ is one-dimensional and that the dimension of $M$ equals the number of irreducible components $n$ of $M = KGM_1$, which are moreover all non-isomorphic. We compute
$$\langle \chi_M, \chi_M \rangle = \left( \begin{array}{cc}
 \dim_K(M)^2 & \dim_K(M_1)^2\\
 0 & \langle \chi_{GM_1}, \chi_{GM_1} \rangle \end{array}\right) = \left( \begin{array}{cc} n^2 & 1 \\
 0 & n \end{array}\right)$$
 For the other direction, the 1 in the upper right corner, entails that $\dim_K(M_1) = 1$. Decompose $GM_1 = \bigoplus_i V_i^{n_i}$, where all the $V_i$ are 1-dimensional ($G$ abelian!). In fact, since $GM_1 \supset M_1$ is irreducible, all the $n_i$ are 0 or 1. It follows that
 $$n = \dim_K(M) \geq \dim_K(GM_1) = \sum_i n_i = \sum_i n_i^2 = n,$$
 whence $M = GM_1$ and $M \supset M_1$ is indeed irreducible.
\end{proof}

Let us deduce the ring structure of $R(\wt{G})$ for $G$ a finite abelian group. Such groups are isomorphic to their character group $\hat{G}$, but the isomorphism is not canonical. By choosing such an isomorphism, one establishes a one-to-one correspondence between the irreducible gliders of essential length $\leq 1$ and zero body and the subsets of $\{ g ~\vline~ g \in G\}$. To see this, we first prove the following lemma.

\begin{lemma}\lemlabel{irlem}
Let $S, T \in \Irr(G)$ be non-isomorphic and let $s,s' \in S, t,t' \in T$. The irreducible gliders $S\oplus T \supset \mathbb{C}(s+t)$ and $S \oplus T \supset \mathbb{C}(s' + t')$ are isomorphic.
\end{lemma}
\begin{proof}
By Schur's lemma, $\End_{KG}(S\oplus T) \mapright{\phi} \mathbb{C}^2$ is an isomorphism. Since $S$ and $T$ are 1-dimensional, there exists $(\lambda, \mu) \in K^2$ such that $s' = \lambda s, t' = \mu t$. Hence $\phi^{-1}(\lambda, \mu)$ defines a glider isomorphism.
\end{proof}

Theorem 3.13 from \cite{CVo2} then shows that the aforementioned one-to-one correspondence is given by
$$ A \in \mc{P}(G) = \mc{P} \longleftrightarrow \bigoplus_{g \in A} T_g \supset Ka,$$
where $T_g$ denotes the irreducible $G$-representation associated to $g \in G$ and $a = \sum_{g \in A} t_g, t_g \in T_g$.  We denote by $\chi_A$ the character associated to the corresponding irreducible glider. For example $\chi_{\{g\}} = \chi_{T_g \supset \mathbb{C}t_g}$. If $A = \emptyset$, then the associated glider is $\mathbb{C} \supset 0$. 
For $g \in G$, we denote
$$\mc{P}^g = \{ A \in \mc{P} ~\vline~ g \in A\}.$$
We also associate $\mathbb{C}$ with $c_\mathbb{C}$, i.e. $\lambda \in \mathbb{C}$ corresponds to the generalized class function $c_\lambda$.

\begin{theorem}\thelabel{charindep}
The glider characters $\{\chi_A~\vline~ A \in \mc{P}\}$ are $\mathbb{C}$-linearly independent.
\end{theorem}
\begin{proof}
Suppose that
$$F = \sum_{A \in \mc{P}} a_A \chi_A = 0, \quad {\rm for~} a_A \in \mathbb{C}.$$
Then $F(1)_{12} = \sum_{A \in \mc{P} \setminus \{\emptyset\}} a_A = 0$. For $g \in G$ we have that $\langle F, \chi_{\{g\}} \rangle = 0$, so in particular
$$ \langle F_{22}, (\chi_{\{g\}})_{22} \rangle^G = \langle F_{22}, \chi_{T_g} \rangle^G = \sum_{A \in \mc{P}^g} a_{A} = 0.$$
Fix an ordering $\{e=g_1,g_2,g_3,\ldots, g_n\}$ of $G$ and choose an ordering $\{A_1, A_2, \ldots, A_{k}\}$ of $\mc{P}$ starting with $\{ \{g_1\}, \{g_2\}, \ldots, \{g_n\}, \{g_{n-1},g_n\}, \ldots\}$. It follows that
$$\underbrace{\left(\begin{array}{c}
C \\
\hline
\begin{array}{cccc} 1 & 1 & \cdots &1 \end{array} \end{array}\right)}_{\wt{C}} \begin{pmatrix} a_{A_1} \\ a_{A_2} \\ \vdots \\ a_{A_k} \end{pmatrix} = \begin{pmatrix} 0 \\ 0 \\ \vdots \\ 0 \end{pmatrix},$$
where $C = (c_{ij})_{ij}$ is an $n \times (2^{|G|}-1)$-matrix defined by
$$c_{ij} = \left\{ \begin{array}{cc}
1 & {\rm if~} g_i \in A_j, \\
0 & {\rm otherwise.} \end{array}\right. $$
Consider the left $(n+1)\times(n+1)$-submatrix of $\wt{C}$:
$$ A =  \left(\begin{array}{ccccc|c}
1 & 0 & \cdots &0& 0 & 0\\
0 & 1 & \cdots &0& 0 & 0\\
\vdots & \vdots & \ddots & \vdots & \vdots & \vdots\\
0 & 0 & \cdots & 1 &0 & 1\\
0 & 0 & \cdots & 0 & 1 & 1\\
\hline
1 & 1 & \cdots &1&1 &1
\end{array}\right)$$
Then 
$$\det(A) = \left| \begin{array}{ccc}
1 & 0 & 1\\
0 & 1 & 1\\
1 & 1 & 1
\end{array} \right| = -1.$$
Hence $\rk(\wt{C}) = n+1$ and it follows that all $a_A = 0$ for $A \neq \emptyset$. Finally, since $\chi_\emptyset(1)_{11} = 1$, it follows that $a_{\emptyset} = 0$ as well.
\end{proof}

It follows that $\{\chi_A~\vline~ A \in \mc{P}\}$ is a $\mathbb{C}$-basis for $R(\wt{G})$. How about the explicit ring structure on $R(\wt{G})$? Let $A, B \in \mc{P}$. Then $\chi_A \chi_B$ is not an irreducible character if and only if there exist $g,g' \in A$, $h,h' \in B$ such that
$$T_g \ot T_{h} \cong T_{g'} \ot T_{h'}.$$
This is equivalent to $gh = g'h'$. It follows that
\begin{equation}\label{multiplication}
\chi_A \chi_B = \chi_C + c\chi_\emptyset,\end{equation}
where $C = \{ gh~\vline~ g \in A, h \in B\}$ and 
$$c = \frac{1}{2}\sum_{\begin{array}{l} g \neq g' \in A\\
h \neq h' \in B\end{array}} \delta_{gh,g'h'}.$$

For example, for $G = C_4 = \{0,1,2,3\}$ the cyclic group of order 4, we have
$$\chi_{\{0,1\}}\chi_{\{1,2\}} = \chi_{\{1,2,3\}} + \chi_{\emptyset}.$$

\begin{proposition}\proplabel{integralextension}
Let $G$ be a finite abelian group. The ring morphism $\iota: R(G) \to R(\wt{G}),~ \chi_g \mapsto \chi_{\{g\}}$ defines an integral ring extension.
\end{proposition}
\begin{proof}
It is clear the $\iota$ defines a ring extension. From \eqref{multiplication} we see that the ideal generated by $\chi_\emptyset$ is just $\mathbb{Z}\chi_\emptyset$ and since $\chi_\emptyset$ is an idempotent, it suffices to show that any $\chi_A$ for $A \in \mc{P}$ is integral over $R(G)[\chi_\emptyset]$. 
Hence, take $A \in \mc{P}$ and write $\chi_A^n = \chi_{A_n} + c_n \chi_\emptyset$ for $n \geq 1$. If $1 \in A$, then 
$$A = A_1 \subset A_2 \subset A_3 \subset \ldots$$
Since $G$ is finite, this chain stabilizes at some $A_n$. Hence it follows that 
$$\chi_A^{n+1} \in R(G)[\chi_\emptyset] + R(G)[\chi_\emptyset]\chi_A + \cdots + R(G)[\chi_\emptyset]\chi_A^n,$$
from which it follows that $R(G)[\chi_\emptyset][\chi_A]$ is a finitely generated $R(G)[\chi_\emptyset]$-module. If $1 \notin A$, take some $g \neq 1 \in A$. It follows that $A_{o(g)}$ contains 1, where $o(g)$ denotes the order of $g$. We know that $\chi_{A}^{o(g)}$ is integral, whence so is $\chi_A$.
\end{proof}

In the following section we continue the study of the ring properties of $R(\wt{G})$ for $G$ finite abelian with chain $1 < G$. To end this section, we consider a slightly more elaborate chain, namely a chain $1 < H < G$ of length 2. For these chains, the irreducible gliders of essential length 2 and zero body look like
$$ \bigoplus_{g \in A} V_g \supset \mathbb{C}Hv \supset \mathbb{C}v$$
for some $A \subset \mc{P}(G)$. In general, if a glider $M \supset M_1 \supset \ldots \supset M_i \supset \ldots$ is irreducible, then so is the glider $M_m \supset M_{m+1} \supset \ldots$ for any $m \geq 0$. In casu, for the glider $\mathbb{C}H \supset \mathbb{C}v$, the group $G$ does not play any role, hence it is also irreducible as $F\mathbb{C}H$-glider and $\mathbb{C}Hv = \oplus_{h \in B} Z_h$ for some subset $B \subset H$. By general glider theory, it is clear that $A$ determines $B$. Concretely, choose an isomorphism $\varphi: G \to \hat{G}$ between $G$ and its character group, and do the same for $H$, i.e. $\psi: H \mapright{\cong} \hat{H}$. The composition morphism
$$ G \mapright{\phi} \hat{G} \mapright{res} \hat{H} \mapright{\psi^{-1}} H$$
is surjective and we denote it by $\pi: G \twoheadrightarrow H$. Let us consider the example $1 \subset C_n \subset C_{nm},$
where $C_{nm} = <a>$ and $C_n = <a^m>$. For a cyclic group $C_N = \langle b \rangle$ we have an isomorphism
$$\phi: C_N \to \hat{C_N},~ b^i \mapsto [ b \mapsto \omega^i],$$
where $\omega$ is an $N$-th primitive root of unity. By composing with the isomorphism $\hat{C_N} \to \mathbb{Z}_N,~ \phi(b^i) \mapsto \ov{i}$ we obtain that 
$$\pi: C_{nm} \twoheadrightarrow{} C_n,~ a^i \mapsto a^{\ov{i}},$$
corresponds to the canonical projection map
$$\pi: \mathbb{Z}_{nm} \twoheadrightarrow{} \mathbb{Z}_n.$$ 
The subsets of $C_N$ correspond to subsets of $\{0,1,\ldots, N-1\}$ so we will work over subsets over the integers, rather than over subsets of the cyclic group. For $A \in \mc{P}(\{0,1,\ldots, nm-1\})$ we denote $\pi(A)$ by $\ov{A} = \{ \pi(\ov{i})~\vline~ i \in A\}$. Hence
$$\mathbb{C}C_nv = \oplus_{i \in \ov{A}} Z_i,$$
where the $Z_i, 0 \leq i < n$ denote the irreducible $C_n$-representations. We introduce the following set
$$C_{n,nm} = \{ (|A|,|\pi(A)|) \in \mathbb{N}^2~\vline~ A \in \mc{P}(\{0,1,\ldots, nm-1\})\}.$$
\propref{ortho} extends to
\begin{proposition}
An $F\mathbb{C}C_{nm}$-glider representation $\Omega \supset M  \supset M_1 \supset M_2 \supset 0 \supset \ldots$ of essential length 2 and zero body is irreducible if and only if 
$$ \langle \chi_M, \chi_M \rangle = \left( \begin{array}{ccc}
k^2 & l^2 & 1 \\
0 & l & l\\
0 & 0& k \end{array} \right)$$
for some $(k,l) \in C_{n,nm}$.
\end{proposition}
As another example, consider the chain $1 \subset C_n \subset C_n \times C_m$. In this case we have a 1-to-1 correspondence between the irreducible gliders and $\mc{P}(\{0,1\ldots, n-1\}) \times \mc{P}(\{0,1,\ldots,m-1\})$. By picking the same isomorphisms as before, the map $\pi: C_n \times C_m \to C_n$ is also just the canonical projection map. Analogously one then proves the following
\begin{proposition}
An $F\mathbb{C}C_nC_m$-glider representation $\Omega \supset M  \supset M_1 \supset M_2 \supset 0 \supset \ldots$ of essential length 2 and zero body is irreducible if and only if 
$$ \langle \chi_M, \chi_M \rangle = \left( \begin{array}{ccc}
(kl)^2 & l^2 & 1 \\
0 & l & l\\
0 & 0& kl \end{array} \right)$$
for some $k \leq m, l \leq n$.
\end{proposition}
\begin{remark}
For $(n,m) = 1$, the group $C_{nm}$ sits in both cases. Let $(k,l) \in C_{n,nm}$, then $l = \pi_{C_n}(k)$. Let $l' = \pi_{C_m}(k)$. Because $n$ and $m$ are coprime, if follows that $k = ll'$ and we see that the results form both propositions are the same.
\end{remark}

\section{primitive central idempotents}\seclabel{primitiveidempotents}

In this section, $G$ is a finite abelian group with chain $1 < G$. From the previous section we know that we can choose an isomorphism between a finite abelian group $G$ and its character group $\hat{G}$. This isomorphism extends automatically to an isomorphism of the group ring $\mathbb{Z}G$ and the representation ring $R(G)$. By tensoring with the rational numbers, we then obtain an isomorphism $\mathbb{Q}G \cong \mathbb{Q} \ot_\mathbb{Z} R(G)$ of $\mathbb{Q}$-algebras. We also provided the ring structure on $R(\wt{G})$ and we observe that $R(\wt{G})$ is the semigroup algebra $\mathbb{Z}S$ of the semigroup $S$ given by
$$S = \{ A \in \mc{P}(G)\}$$ with multiplication
as defined in \eqref{multiplication}. For generalities on semigroups, we refer to \cite{Ok}.\\

In the previous section, we derived an integral ring extension $R(G) \hookmapright{} R(\wt{G})$. This allows us to consider the $\mathbb{Q}$-algebra extension $\mathbb{Q}G \hookmapright{} \mathbb{Q} \ot_{\mathbb{Z}} R(\wt{G})$ ($\mathbb{Q}$ is flat as $\mathbb{Z}$-module). We denote $\mathbb{Q}(\wt{G}) :=  \mathbb{Q} \ot_{\mathbb{Z}} R(\wt{G})$. Since $G$ is finite, Maschke's theorem entails that the rational group algebra $\mathbb{Q}(G)$ is semisimple, hence our first task is to determine the Jacobson radical of $\mathbb{Q}(\wt{G})$. Secondly, we will determine the primitive central idempotents. To this extent, we mention the results obtained in \cite{JLP}. In loc. cit. the authors describe the primitive central idempotents of the group algebra $\mathbb{Q}G$ for $G$ a finite nilpotent group. Let us introduce some notation. Let $H \leq G$ be a subgroup. Then $\hat{H} = \frac{1}{|H|} \sum_{h\in H} h$ is an idempotent, which is central if and only if $H$ is normal. Denote by $\mc{M}(G)$ the set of all minimal normal subgroups and define
$$\epsilon(G) = \Pi_{L \in \mc{M}(G)} (1 - \hat{L}).$$
For a normal subgroup $N \triangleleft G$, we consider $\mc{M}(G/N)$ and write $L$ for the associated subgroup to $\ov{L}$ of $G/N$ containing $N$. It appears that 
$$\epsilon(G,N) = \hat{N} \Pi_{\ov{L} \in \mc{M}(G/N)}(1- \hat{L})$$ 
is a primitive central idempotent if and only if $G/N$ is cyclic and these are all the primitive central idempotents, see \cite[Corollary 2.1]{JLP}.\\

In proving their results in \cite{JLP}, the authors used the isomorphism $\mathbb{Z}(G/H) \cong (\mathbb{Z}G)\hat{H}$, for $H$ a normal subgroup in $G$. In our setting, we can extend the canonical homomorphism $\omega: G \to G/H$ to a $\mathbb{Z}$-linear morphism $\omega_H: R(\wt{G}) \to R(\wt{G/H})$, defined by
$$\omega_{H}(\sum_{A \in \mc{P}(G)} a_A \chi_A) = \sum_{A \in \mc{P}(G)} a_A \chi_{\omega(A)}.$$
This morphism is however not a ring morphism in general.
\begin{example}
Consider the groups $\mathbb{Z}_2 \subset \mathbb{Z}_4$ and write $\mathbb{Z}_4/\mathbb{Z}_2 = \{e,a\}$. Then in $R(\wt{\mathbb{Z}_4})$ we have $\chi_{\{1,3\}}\chi_{\{0,2\}} = \chi_{\{1,3\}} + 2\chi_{\emptyset}$ but in $R(\wt{\mathbb{Z}_4/\mathbb{Z}_2})$ it holds
$$\chi_{\{a\}}\chi_{\{1\}} = \chi_{\{a\}}.$$
\end{example}
By modding out the ideal $(\chi_\emptyset) = \mathbb{Z}\chi_\emptyset$ we get the induced map
$$\ov{\omega}_H: R(\wt{G})/(\chi_\emptyset) \to R(\wt{G/H})/(\chi_{\emptyset}).$$
\begin{lemma}
Let $H \subset G$ be a subgroup. The map $\ov{\omega}_H$ is a ring morphism.
\end{lemma}
\begin{proof}
Trivial.
\end{proof}

Hence, from now on we will always work modulo the ideal $(\chi_\emptyset)$ but we do not alter our notation, i.e. $R(\wt{G})$ means $R(\wt{G})/(\chi_\emptyset)$. In fact, from a semigroup point of view, the element $\chi_\emptyset$ is a 0-element. To get some feeling with these semigroup algebras, we include the following example

\begin{example}\exalabel{Cpfirst}
For $G = C_p = \langle a \rangle$, the cyclic group of order $p$, we have that
$$1 = \epsilon(G) + \hat{G} \in \mathbb{Q}(C_p).$$
In $\mathbb{Q}(\wt{C_p})$ however, one verifies that 
$$\psi(C_p) := \sum_{i=0}^{p-1} \frac{1}{p} \chi_{\{a^i\}} - \chi_{C_p}$$
is idempotent and
$$1 = \chi_{\{1\}} = \epsilon(C_p) + \psi(C_p) + \chi_{C_p},$$
is a decomposition into orthogonal idempotents. For now, it is not clear whether $\psi(C_p)$ is primitive or not.
\end{example}

Let us compute the kernel $\Ker(\omega_H)$. Choose a transversal set $T$ for $H$ in $G$. Every element $g$ in $G$ can be written uniquely as the product $th$ for some $t \in T$, $h \in H$, hence we can associate to any $A \subset G$, the subset $T_A \subset T$ consisting of all the $t$'s appearing in these products. Define
$$\mc{P}(G,B) = \{ A \in \mc{P}(G)~\vline~ T_A = B\}, \quad B \in \mc{P}(T).$$
Clearly, it holds that $\mc{P}(G) = \sqcup_{B \in \mc{P}} \mc{P}(G,B)$ and in particular the element $\chi_{BH} \in \mc{P}(G,B)$. Let $\alpha = \sum_{A \in \mc{P}(G)} a_A\chi_A \in R(\wt{G})$. Define
$$\alpha' = \sum_{B \in \mc{P}(T)} \big( \sum_{A \in \mc{P}(G,B)} a_A\big)\chi_{BH}.$$
We have that $\omega_H(\alpha) = \omega_H(\alpha')$. Suppose that $\omega_H(\alpha) = 0$. By \theref{charindep} it follows that for any $B \in \mc{P}(T)$ it holds
$$a_{BH} = - \sum_{A \in \mc{P}(G,B) \setminus \{BH\}} a_{A}.$$
It follows that
\begin{eqnarray*}
\alpha &=& \sum_{B \in \mc{P}(T)} \Bigg[ \sum_{A \in \mc{P}(G,B) \setminus \{BH\}} -a_A\chi_{BH} + a_A\chi_A\Bigg]\\
&=& \sum_{B \in \mc{P}(T)} \Bigg[ \sum_{A \in \mc{P}(G,B) \setminus \{BH\}} -a_A\chi_A\chi_H + a_A\chi_A\Bigg]\\
&=& \sum_{B \in \mc{P}(T)} \Bigg[ \sum_{A \in \mc{P}(G,B) \setminus \{BH\}} a_A\chi_A(1-\chi_H)\Bigg].
\end{eqnarray*}
Hence we have just proven the following 
\begin{proposition}\proplabel{ke}
The kernel of $\omega_H$ equals $\Ker(\omega_H) = (\chi_H - 1)$.
\end{proposition}
As a corollary we obtain that $R(\wt{G/H}) \cong R(\wt{G})\chi_H$, hence $\mathbb{Q}(\wt{G/H}) \cong \mathbb{Q}(\wt{G})\chi_H$.
\begin{proposition}
The element $\chi_G \in \mathbb{Q}(\wt{G})$ is a primitive central idempotent.
\end{proposition}
\begin{proof}
The kernel of the ring morphism $~\cdot ~\chi_G: \mathbb{Q} \ot_\mathbb{Z} R(\wt{G}) \to   \mathbb{Q} \ot_\mathbb{Z} R(\wt{G})\chi_G$ is the annihilator of $\chi_G$. Hence by \propref{ke} we have that
$$  \mathbb{Q} \ot_\mathbb{Z} R(\wt{G})\chi_G \cong   \mathbb{Q} \ot_\mathbb{Z} R(\wt{G})/\Ker(\omega_G) \cong  \mathbb{Q} \ot_\mathbb{Z} \mathbb{Z} \cong \mathbb{Q}.$$
\end{proof}

Again in terms of semigroups, the element $\theta = \chi_G$ is a 0-element. From now, we consider the contracted semigroup algebra $\mathbb{Q}S_\theta$. It is known that any cyclic semigroup $\langle s \rangle$ in a finite semigroup $S$ contains an idempotent. There is also a natural order on the set $E(S)$ of idempotents given by $e \leq f$ if and only if $e = ef = fe$ for $e,f \in E(S)$.
\begin{lemma}
Let $A \in \mc{P}(G)$, then the cyclic semigroup $\langle \chi_A \rangle$ contains a unique minimal idempotent.
\end{lemma}
\begin{proof}
It is clear that $\chi_B^2 = \chi_B$ if and only if $B \leq G$ is a subgroup. Suppose that $A^n = H$ and $A^m = H'$ are both idempotents, i.e. subgroups of $G$ and suppose that $n \leq m$. We can write $m = kn + r$ for some $0 \leq r \leq m-1$. It follows that $H' = A^m = HA^{m-n} = HA^{(k-1)n}A^r = HA^r$. Hence $(HA^r)^2 = HA^{2r} = HA^r$. From this it follows that
$$A^m = (HA^r)^2 = (HA^r)^n = H = A^n.$$
It follows that there is only one idempotent in $\langle A \rangle$ which is therefore minimal.
\end{proof}
\begin{definition}
For $A \in \mc{P}(G)$ we define $n(A)$ to be the unique minimal idempotent in $\langle A \rangle$.
\end{definition}
\begin{lemma}\lemlabel{class}
Let $A \in \mc{P}(G)$ and $n(A) = H$. Then either $A \subset H$ or $A \cap H = \emptyset$. 
\end{lemma}
\begin{proof}
Suppose that $b \in A \cap H$ and let $n \geq 0$ be such that $A^n = H$. Let $a \in A$, then $ab^{n-1} \in H$. Since $b \in H$, so is $b^{1-n}$, whence $a \in H$. This shows that $A \subset H$.
\end{proof}
\begin{lemma}
Let $A \in \mc{P}(G)$ and $n(A) = H$. Then $A \subset gH$ for some $g \in G$.
\end{lemma}
\begin{proof}
If $A \cap H \neq \emptyset$ then $g = 1$ satisfies by the previous lemma, so suppose that $A \cap H = \emptyset$. Let $n \geq 0 $ be such that $A^n = H$, then we have that $A \subset a^{1-n}H$ for some $a \in A$.
\end{proof}

We can now construct nilpotent elements in the (contracted) semigroup algebra $\mathbb{Q}(\wt{G})$. Let $A \in \mc{P}(G)$ with $n(A) = H$. By \lemref{class} there exists some $g \in G$ such that $A \subset gH$. The element $\chi_{gH} - \chi_A \in J$ where $J = J(\mathbb{Q}(\wt{G}))$ denotes the Jacobson radical. Indeed, suppose first that $g = 1$. For $m$ the smallest integer such that $A^m = H$ we have
\begin{eqnarray*}
(\chi_H - \chi_A)^m &=& \sum_{i=1}^m {{m}\choose{i}}(-1)^{m-i} \chi_H^i \chi_A^{m-i}\\
&=&  \sum_{i=1}^m { {m}\choose{i}}(-1)^{m-i} \chi_H\\
&=& (1 -1)^m \chi_H = 0.
\end{eqnarray*}
Now if $g \notin H$, then $g^{-1}A \subset H$ and $\chi_H - \chi_{g^{-1}A} \in J$. It follows that 
$$\chi_{gH} - \chi_A = \chi_g(\chi_{H} - \chi_{g^{-1}A}) \in J.$$
Denote by $I$ the ideal generated by the elements $\chi_{gH} - \chi_A$ where $A$ runs over the subsets of $G$. Observe that if $A = \{g\}$ is a singleton, that $n(A) = 1$ and $\chi_{\{g\}1} - \chi_{A} = 0$. 
\begin{theorem}\thelabel{Ja}
The Jacobson radical $J$ equals $I$ and we have a ring isomorphism
$$\varphi: \bigoplus_{H < G} \mathbb{Q}(G/H) \to \mathbb{Q}(\wt{G})/J,$$
where the unit element $e_H$ of $\mathbb{Q}(G/H)$ is sent to $\Pi_{\ov{H'} \in \mc{M}(G/H)} (\chi_H - \chi_{H'}).$
\end{theorem}
\begin{proof}
In order for $\phi$ to be well-defined, it suffices to check that $\varphi(e_H)\varphi(e_{Z}) = 0$ for two different subgroups $H, Z$ of $G$. To see this, choose subgroups $\ov{H'} \in \mc{M}(G/H), \ov{Z'} \in \mc{M}(G/Z)$ such that $H' \leq HZ$ and $Z' \leq HZ$. In the product $\varphi(e_H)\varphi(e_Z)$ there appear the factors $(\chi_H - \chi_{H'})$ and $(\chi_Z - \chi_{Z'})$. We calculate that their product equals
\begin{equation}\label{jo}
\chi_{HZ} - \chi_{H'Z} - \chi_{HZ'} + \chi_{H'Z'}.
\end{equation}
We have the inclusions $H'Z \leq HZ \leq H'Z$ and $HZ' \leq HZ \leq HZ'$, or $HZ = H'Z = HZ'$. From this it follows that
$$H'Z' = H'HZ' = H'HZ = H'Z = HZ,$$
and we see that \eqref{jo} is indeed 0. To show that $\varphi$ is surjective, we observe that $\varphi(e_H) = \chi_H$ for $H < G$ a maximal subgroup. If $n = l(G)$ is the Jordan-H\"older length of $G$, then for an $H$ of length $n-2$ we have that
$$\varphi(e_H) = \chi_H + \sum_{\ov{H'} \in \mc{M}(G/H)} a_{H'}\chi_{H'} \in \Ima(\varphi).$$
Surjectivity now follows from a downwards induction on the Jordan-H\"older lengths of the subgroups. That $\varphi$ is an isomorphism now follows from a dimension argument and the claim about the Jacobson radical is a direct consequence of Wedderburn-Artin.
\end{proof}
As a corollary we can give a complete list of the primitive central idempotents in $\mathbb{Q}(\wt{G})$. To this extent we denote 
$$\psi(G,H) := \varphi(e_H) = \Pi_{\ov{H'} \in \mc{M}(G/H)} (\chi_H - \chi_{H'}).$$
Also, for $A \subset G$ we have the associated element $\widehat{A} = \frac{1}{|A|}\sum_{g \in A}g$ in $\mathbb{Q}(G)$. We denote the element $\frac{1}{|A|}\sum_{g \in A} \chi_{\{g\}}$ in $\mathbb{Q}(\wt{G})$ by $\check{A}$. Of course, there is an embedding $\mathbb{Q}(G) \subset \mathbb{Q}(\wt{G})$, but this does not correspond to the embedding given by $\varphi$. Finally, by $\check{\epsilon}(G,N)$ we denote the element
$$\check{\epsilon}(G,N) = \check{N} \Pi_{\ov{L} \in \mc{M}(G/N)}(1- \check{L}) \in \mathbb{Q}(\wt{G}).$$

Let $H <G$ be a subgroup and take a subgroup $\ov{N}$ in $G/H$. The group $\ov{N}$ is isomorphic to some $N/H$ with $H \leq N \leq G$ a subgroup. Pick a transversal set $T^N_H$ for $H$ in $N$. Inside the group algebra $\mathbb{Q}(G/H)$ we have the element
$$\widehat{\ov{N}} = \frac{1}{|T^N_H|} \sum_{t \in T^N_H} e_{tH}.$$
It follows that
$$\varphi(\widehat{\ov{N}}) = (\frac{1}{|T^N_H|} \sum_{t \in T^N_H} \chi_{\{t\}}) \psi(G,H) = \check{T^N_H}\psi(G,H)$$
Let $e$ be a primitive idempotent in $\mathbb{Q}(\wt{G})$. By the results in \cite{JLP} and \theref{Ja} 
$$\ov{e} = \varphi(\epsilon(G/H,\ov{N}))$$
 for subgroups $H \leq N \leq G$ such that $G/N$ is cyclic. We calculate
\begin{eqnarray*}
\varphi(\epsilon(G/H,\ov{N})) &=& \varphi(\widehat{\ov{N}}) \Pi_{[\ov{L}] \in \mc{M}(\frac{G/H}{\ov{N}})}(\varphi(e_H) - \varphi(\widehat{\ov{L}}))\\
&=&  \check{T^N_H} \psi(G,H) \Pi_{[\ov{L}] \in \mc{M}(\frac{G/H}{\ov{N}})} \big( \psi(G,H) - \check{T^L_H}\psi(G,H)\big)\\
&=& \psi(G,H)\check{T^N_H} \Pi_{\ov{L} \in \mc{M}(G/N)}(1 - \check{T^L_H})\\
&=& \psi(G,H)\check{\epsilon}(G,N).
\end{eqnarray*}

\begin{theorem}\thelabel{primitiveidempotents}
Let $G$ be a finite abelian group. The primitive central idempotents of $\mathbb{Q}(\wt{G})$ are precisely the elements $\psi(G,H)\check{\epsilon}(G,N)$ for subgroups $H \leq N \leq G$ such that $G/N$ is cyclic.
\end{theorem}

\section{A more complicated chain}\seclabel{complicatedchain}

One of the ideas of glider representation theory is to build a new sort of representation theory which captures the information given by the algebra or ring filtration. Therefore, we will now consider the generalized character ring $R(\wt{G})$ for the slightly more complicated chain $1 \subset H  \subset G$, where $G$ is finite abelian. From now on, we will denote this ring by $R(H < G)$. When we write $R(\wt{G})$ below, we mean the character ring associated to the chain $1 \subset G$. From the discussions after \propref{integralextension} we know that the irreducible gliders of essential length 2 are determined by $\{ (A, \pi(A)) \in \mc{P}(G) \times \mc{P}(H)\}$. Recall that $\pi: G \twoheadrightarrow H$ is the surjective group morphism obtained by choosing group isomorphisms between $G, H$ and their character groups. Denote by $K = \Ker(\pi)$.
\begin{example}
Let $V_4 = \langle a,b ~\vline~ a^2 = 1 = b^2 \rangle$ and denote $c = ab$. Let $H = \{1, c\}$. If we denote by $T_g$ ($g = a, b ,c$) the irreducible representation which sends $g$ and $1$ to $1$ and the other two elements to -1, then $\pi: V_4 \twoheadrightarrow \{1,c\}$ is given by
$$1 \mapsto 1,~ a \mapsto c, ~ b \mapsto c, ~ c \mapsto 1,$$
whence $K = \{ 1,c\}$. If we would have chosen the isomorphism $V_4 \cong \hat{V}_4, a \mapsto T_c, b \mapsto T_a, c \mapsto T_b$, then $K$ would be $\{1,a\}$.
\end{example}
For $A \in \mc{P}(G)$ we denote the associated element in $R(H < G)$ by $\chi_{(A,\pi(A))}$.\\

When we dealt with the chain $1 \subset G$, we only had to look at irreducible gliders of essential length 1 or 0, the latter corresponding to $\mathbb{C} \supset 0 \supset \ldots$ which yielded a 0-element in $R(\wt{G})$. For $1 \subset H \subset G$, we also have to consider irreducible gliders of essential length 1, which are of the form $\Omega \supset M \supset \mathbb{C}m \supset 0 \supset \ldots$, with $M = \mathbb{C}Hm$. It follows that $M \supset M \supset \mathbb{C}m$ is a $\mathbb{C} \subset \mathbb{C}H$-glider. Of course, given $\Omega$, we can associate to it the irreducible $F\mathbb{C}G$-glider
$$\Omega \supset \mathbb{C}Gm \supset M \supset \mathbb{C}m \supset 0 \supset \ldots$$
and we see that different $\Omega's$ yield different gliders. In any case, we see that the irreducible gliders of essential length $1$ correspond to the irreducible gliders for $1 \subset H$, which we know are labeled by the subsets of $H$. For $B \in \mc{P}(H)$ we write $\chi_{(\emptyset, B)}$ for the element in $R(H<G)$. Finally, the irreducible glider $\mathbb{C} \supset 0 \supset \ldots$ of essential length 0 yields the element $\chi_{(\emptyset,\emptyset)}$.\\

Let us determine the multiplication rules in $R(H < G)$. First, let $M \supset M_1 \supset \mathbb{C}m \supset 0 \supset \ldots$ and $N \supset N_1 \supset \mathbb{C}n \supset 0 \supset \ldots$ be two irreducible gliders of essential length 2 corresponding to $A$ and $B$ in $\mc{P}(G)$. In the fragment decomposition of $M \ot N$ there appears only one irreducible glider of essential length 2, namely
$$ \bigoplus_{c \in AB} T_c \supset \bigoplus_{d \in \pi(AB)} S_d \supset \mathbb{C} m \ot n,$$
where $T_a$ denotes an irreducible representation of $G$, and $S_b$ one of $H$. As $H$-representations, we have 

$$M_1 \ot N_1 \cong (\bigoplus_{h \in \pi(A)}S_h) \ot (\bigoplus_{h' \in \pi(B)} S_{h'}) \cong \bigoplus_{c \in \pi(AB)} S_c^{n_c},$$
for some $n_c \geq 1$. In fact, whenever $n_c >1$, we have $n_c - 1$ irreducible gliders of essential length 1 splitting off. In fact, these gliders are isomorphic to $\chi_{(\emptyset, \{c\})}$. Finally, in the decomposition of $M \ot N$ there can appear irreducible gliders of essential length $0$, from which it follows that
$$\chi_{(A,\pi(A))}\chi_{(B,\pi(B))} = \chi_{(AB, \pi(AB))} + \sum_{c \in \pi(AB)} (n_c-1)\chi_{(\emptyset, \{c\})} + d \chi_{(\emptyset,\emptyset)},$$
for some $d \geq 0$.
Next, we consider the product of $\chi_{(A,\pi(A))}$ and $\chi_{(\emptyset,B)}$ for $A \subset G$, $B \subset H$. Their tensor product becomes
$$M \ot N \supset M_1 \ot \mathbb{C}n \supset 0 \supset \ldots$$
It follows that
$$\chi_{A,\pi(A))}\chi_{(\emptyset,B)} = \sum_{h \in \pi(A)} \chi_{(\emptyset, Bh)} + d\chi_{(\emptyset,\emptyset)},$$
for some $d \geq 0$. Finally, it is straightforward that 
$$\chi_{(\emptyset,B)}\chi_{(\emptyset,C)} = \chi_{(\emptyset, BC)} + d\chi_{(\emptyset,\emptyset)},$$
for $B,C \in \mc{P}(H)$ and some $d \geq 0$.\\

Combining everything we see that $R(H <G)$ is the $\mathbb{Z}$-algebra generated by the elements of  
\begin{equation}\nonumber
\resizebox{0.9 \hsize}{!}{$S = \{ \chi_{(A,\pi(A))} ~\vline~ (A,\pi(A)) \in \mc{P}(G) \times \mc{P}(H)\} \cup \{ \chi_{(\emptyset, B)}~\vline~(\emptyset, B), B \in \mc{P}(H)\} \cup \{\chi_{(\emptyset,\emptyset)}\}.$}
\end{equation}
The element $\theta = \chi_{(\emptyset,\emptyset)}$ is again a 0-element hence we can work modulo the two-sided ideal $\mathbb{Z}\chi_{(\emptyset,\emptyset)}$ and we obtain the following relations
\begin{eqnarray}
 \chi_{(A,\pi(A))}\chi_{(B,\pi(B))} &=& \chi_{(AB, \pi(AB))} + \sum_{c \in \pi(AB)} (n_c-1)\chi_{(\emptyset, \{c\})}, \label{rela}\\
 \chi_{A,\pi(A))}\chi_{(\emptyset,B)} &=& \sum_{h \in \pi(A)} \chi_{(\emptyset, Bh)},\nonumber\\
 \chi_{(\emptyset,B)}\chi_{(\emptyset,C)} &=& \chi_{(\emptyset, BC)}. \nonumber
 \end{eqnarray}

We again wonder whether we can find nilpotent and idempotent elements in 
$$\mathbb{Q}(H<G) : = \mathbb{Q} \ot_{\mathbb{Z}} R(H<G)$$ 
(in fact we work module the two-sided ideal $\mathbb{Q}\chi_{(\emptyset,\emptyset)}$). From the multiplication rules above, we see that the element $\chi_{(\emptyset,\{1\})}$ is a central idempotent and one verifies that we obtain a ring isomorphism
\begin{equation}\label{firstiso}
\mathbb{Q}(H <G) \mapright{\cong} \mathbb{Q}(H<G)(1-\chi_{(\emptyset,\{1\})}) \times \mathbb{Q}(\wt{H}),
\end{equation}
which is defined by
$$\chi_{(A,\pi(A))} \mapsto (\chi_{(A,\pi(A))} - \sum_{h \in \pi(A)} \chi_{(\emptyset,\{h\})}, \sum_{h \in \pi(A)} \chi_{(\emptyset,\{h\})})$$
and
$$\chi_{(\emptyset,B)} \mapsto (0, \chi_{(\emptyset,B)}).$$
We already know what the Jacobson radical of $\mathbb{Q}(\wt{H})$ is, so we focus on \newline$T: =  \mathbb{Q}(H<G)(1-\chi_{(\emptyset,\{1\})})$. Denote by $f: \mathbb{Q}(H <G) \to T$ the ring morphism obtained by composing the above isomorphism with the projection on $T$. Because 
$$f(\chi_{(A,\pi(A))})f(\chi_{(B,\pi(B))}) = f(\chi_{(AB,\pi(A)\pi(B))}),$$
the calculations from \secref{primitiveidempotents} that lead to the nilpotent elements $\chi_A - \chi_{gn(A)}$ if $A \subset gn(A)$ go through here. Explicitly, for $A \subset G$ there exist $g \in A$ and a subgroup $N \leq G$ such that $A \subset gN$ and then $f(\chi_{(A,\pi(A))} - \chi_{(gN,\pi(g)\pi(N))})$ is nilpotent in $S$. If we denote by $I$ the ideal of $S$ generated by these elements we arrive at 
\begin{proposition}
We have an isomorphism of rings
$$\bigoplus_{N < G} \mathbb{Q}(G/N) \to \frac{\mathbb{Q}(H < G)(1 - \chi_{(\emptyset,\{1\})})}{I},$$
where the unit element $e_N$ of $\mathbb{Q}(G/N)$ is sent to $\Pi_{\ov{N'} \in \mc{M}(G/N)} (f(\chi_{(N, \pi(N))})-f(\chi_{(N',\pi(N'))}))$.
\end{proposition} 
\begin{proof}
By the discussion above, the proof is analogous to the proof of \theref{Ja}.
\end{proof}

 \begin{remark}
We observe that the elements $\chi_{(\emptyset,B)}, B \in \mc{P}(H)$ generate an ideal $P_H$ and we have the chain of ideals
 $$\mathbb{Q}\chi_{(\emptyset,\emptyset)} \triangleleft P_H \triangleleft \mathbb{Q}(H<G).$$
 From the defining relations it follows that the quotient $\mathbb{Q}(H<G)/P_H$ is isomorphic to $\mathbb{Q}(e < G)$ and the ideal $P$ itself is isomorphic (as algebra without considering the unit) to $\mathbb{Q}(e<H)$. Consider now a chain of three abelian groups $e < G_1 < G_2 < G_3 = G$. We can label the irreducible gliders of essential length 3 by $\chi_{(A,\pi_1(A),\pi_2(A))}$ for $A \in \mc{P}(G)$ and $\pi_1: G_3 \to G_2, \pi_2: G_2 \to G_1$ fixed epimorphisms. One can write down the defining relations which are similar to \eqref{rela}. The element $\chi_{(\emptyset,\emptyset,\{1\})}$ is again a central idempotent and we have a similar ring isomorphism as in \eqref{firstiso}
 $$\mathbb{Q}(G_1<G_2<G_3) \cong \mathbb{Q}(G_1<G_2<G_3)(1 - \chi_{(\emptyset,\emptyset,\{1\})}) \times \mathbb{Q}(e< G_1).$$
Again similar to the above calculations, one checks that $\mathbb{Q}(G_1<G_2<G_3)(1 - \chi_{(\emptyset,\emptyset,\{1\})}) $ is isomorphic to $\mathbb{Q}(G_2<G_3)$. By induction one then shows that for any chain $e < G_1 < \ldots < G_d = G$ of abelian subgroups we have that
$$\mathbb{Q}(G_1<\ldots < G_d) \cong \mathbb{Q}(e<G_d) \times \mathbb{G}(e<G_{d-1}) \times \cdots \times \mathbb{Q}(e<G_1)$$
and this corresponds to a chain of ideals
$$\mathbb{Q}\chi_{(\emptyset,\ldots,\emptyset)} \triangleleft P_{G_1} \triangleleft \ldots \triangleleft P_{G_{d-1}} \triangleleft \mathbb{Q}(G_2<\ldots < G_d).$$
 \end{remark}

\section{The quaternion group $Q_8$}

In this section we compute the generalized character ring for the non-abelian group \newline$Q_8 = \langle i,j,k ~\vline ~i^2 = j^2 = k^2 = ijk \rangle$ with chain $1 \subset Q_8$. The character table is given by
$$\begin{array}{c|ccccc}
& \{1\} & \{-1\} & \{i,-i\} & \{j,-j\} & \{k,-k\} \\
\hline
T_1 & 1 & 1 & 1 & 1 & 1\\
T_i & 1 & 1 & 1 & -1 & -1\\
T_j & 1 & 1 & -1 & 1 & -1\\
T_k & 1 & 1 & -1 & -1 & 1\\
U & 2 & -2 & 0 & 0 & 0
\end{array}$$
The irreducible characters associated to the one-dimensional representations generate an abelian subgroup of the character ring $R(Q_8)$ isomorphic to $V_4 = \langle a,b~\vline~ a^2 = b^2 \rangle$. We fix the isomorphism
$$1 \mapsto T_1 \quad a \mapsto T_i \quad b \mapsto T_j \quad c \mapsto T_k$$
To discuss the behavior of the two-dimensional irreducible representation $U$ we fix a basis $\{e_1,e_2\}$ such that $U$ has the following presentation:
$$i \mapsto \begin{pmatrix} 0 & i \\ i  & 0\end{pmatrix},~ j \mapsto \begin{pmatrix} - i & 0 \\ 0 & i \end{pmatrix}.$$
Theorem 3.13 from \cite{CVo2} together with Schur's lemma show that every point $[\lambda:\mu] \in \mathbb{P}^1$ determines an irreducible glider representation $U \supset \mathbb{C}(\lambda e_1 + \mu e_2)$ and two different points in $\mathbb{P}^1$ yield non-isomorphic gliders. Since $U$ is 2-dimensional, there also exist irreducible gliders of the form $U^{\oplus 2} \supset \mathbb{C}(u_1 + u_2)$ such that $\dim_{\mathbb{C}}(<u_1,u_2>) = 2$. By the following lemma, there is, up to isomorphism, only one such irreducible glider.
\begin{lemma}
Let $M \supset M_1 \supset 0 \supset \ldots, N \supset N_1 \supset 0 \supset \ldots$ be irreducible gliders with $M \cong U^{\oplus 2} \cong N$. Then both gliders are isomorphic.
\end{lemma}
\begin{proof}
There exist $u_i, u_i' \in U$ ($i=1,2$) such that $M_1 = \mathbb{C}(u_1 + u_2), N_1 = \mathbb{C}(u'_1 + u'_2)$. Let $B$ be a base change matrix for $\{u_1,u_2\}$ and $\{u'_1,u'_2\}$. Since $\End_{Q_8}(U^{\oplus 2}) \cong M_2(\mathbb{C})$, we have that $B$ yields an isomorphism between both gliders.
\end{proof}
As a corollary we deduce that the irreducible gliders of essential length 1 and zero body are labeled by
$$\mc{P}(V_4) \times (\mathbb{P}^1 \sqcup \ast),$$
where $\ast$ indicates that $U^{\oplus 2}$ appears in the decomposition of $M$. If an irreducible character has label $(A,[\lambda:\mu])$, we denote the associated generalized character by $\chi_{(A,[\lambda: \mu])}$ as we did in the abelian case. For example, the character of 
$$T_1 \oplus T_b \oplus U^{\oplus 2} \supset \mathbb{C}(t_1 + t_b + u_1 + u_2) \supset 0 \supset \ldots$$
is denoted by $\chi_{(\{1,b\},\ast)}$.\\

Let us now compute the multiplication rules in $R(\wt{Q_8})$. First of all, the character associated to the irreducible glider $\mathbb{C} \supset 0$ is again denoted by $\chi_{(\emptyset,\emptyset)}$, which is a $0$-element. In what follows, we will always set this element equal to 0. We wonder what happens to the tensor product
\begin{equation}\label{u}
\big[U \supset \mathbb{C}(\lambda_1e_1 + \mu_1e_2) \big]\ot \big[U \supset \mathbb{C}(\lambda_2e_1 + \mu_2e_2)\big].
\end{equation}
We have the following classical result
\begin{proposition}
Let $G$ be a finite group, $V$ an irreducible representation. Then $V \ot V$ is irreducible if and only if $V$ is one-dimensional.
\end{proposition}
\begin{proof}
If $\dim_\mathbb{C}(V) > 1$, then the switch map $\tau : V\ot V \to V \ot V,~v\ot w \mapsto w \ot v$ is not a scalar multiplication of the identity morphism, hence $\dim_\mathbb{C}\End_G(V\ot V) >1$, i.e. $V \ot V$ is reducible.  
\end{proof}
We know that $U \ot U = S(U \ot U) \oplus A(U \ot U)$ decomposes into a symmetric and antisymmetric part, from which it follows that $\chi_U = \chi_S + \chi_A$, see e.g. \cite[Chapter 19]{JL}. Here, $\chi_U$ denotes the associated character of $U$, i.e. $\chi_U \in R(Q_8)$. One checks that $\chi_A = \chi_{T_1}$ and $\chi_S = \chi_{T_i} + \chi_{T_j} + \chi_{T_k}$. A decomposition is given by
\begin{eqnarray}
U \ot U &=& \underbrace{\mathbb{C}(e_1 \ot e_2 - e_2 \ot e_1)}_{T_1} \oplus \underbrace{\mathbb{C}(e_1 \ot e_2 + e_2 \ot e_1)}_{T_j}\nonumber \\ \label{decoUU}
&& \oplus \underbrace{\mathbb{C}(e_1 \ot e_1 + e_2 \ot e_2)}_{T_k} \oplus \underbrace{\mathbb{C}(e_1 \ot e_1 - e_2 \ot e_2)}_{T_i}.
\end{eqnarray}
Let $[\lambda_i:\mu_i] \in \mathbb{P}^1,~i=1,2$, then $(\lambda_1e_1+\mu_1e_2) \ot (\lambda_2e_1 + \mu_2e_2)$ decomposes into
$$a(e_1 \ot e_2 - e_2 \ot e_1) + b(e_1 \ot e_2 + e_2 \ot e_1) + c(e_1 \ot e_1 + e_2 \ot e_2) + d(e_1 \ot e_1 - e_2 \ot e_2).$$
and we obtain that \eqref{u} remains irreducible if and only if $abcd \neq 0$. Since
\begin{equation}\label{ju}
\begin{pmatrix} \lambda_1\lambda_2 & \lambda_1\mu_2 & \mu_1\lambda_2 & \mu_1\mu_2 \end{pmatrix} = \begin{pmatrix} c+d & a+b & b-a & c-d \end{pmatrix}
\end{equation}
this is equivalent to 
$$ \left| \begin{array}{cc}
\lambda_1 & \pm \lambda_2 \\
\mu_1 & \mu_2 \end{array}\right| \neq 0 {\rm~and~} \left| \begin{array}{cc}
\lambda_1 & \pm \mu_1 \\
\mu_2 & \lambda_2 \end{array}\right| \neq 0.$$
\begin{remark}
The above is independent of the choice of base. Indeed, suppose that $\begin{pmatrix}e_1&e_2 \end{pmatrix} = \begin{pmatrix}f_1&f_2 \end{pmatrix}B$ is a base change, then 
$$\begin{pmatrix}e_1 \ot e_1 & e_1 \ot e_2 & e_2 \ot e_1 & e_2 \ot e_2 \end{pmatrix} = \begin{pmatrix}f_1 \ot f_1 & f_1 \ot f_2 & f_2 \ot f_1 & f_2 \ot f_2 \end{pmatrix}B \ot B,$$
where $B \ot B$ denotes the Kronecker tensor product of two matrices. With regard to the basis $\{f_1,f_2\}$ equation \eqref{ju} becomes
$$\begin{pmatrix} \lambda_1\lambda_2 & \lambda_1\mu_2 & \mu_1\lambda_2 & \mu_1\mu_2 \end{pmatrix}(B\ot B)^T = \begin{pmatrix} c+d & a+b & b-a & d-c \end{pmatrix}(B\ot B)^T,$$
and $B \ot B$ is invertible.
\end{remark}
Generically, we have
$$\chi_{(\emptyset,[\lambda_1:\mu_1])}\chi_{(\emptyset,[\lambda_2:\mu_2])} = \chi_{(V_4,\emptyset)}.$$
If, for example, $\lambda_1\mu_2 = \mu_1\lambda_2$, but $\lambda_1\lambda_2 \neq \pm \mu_1\mu_2$ then $\chi_{(\emptyset,[\lambda_1:\mu_1])}\chi_{(\emptyset,[\lambda_2:\mu_2])} = \chi_{(\{a,b,c\},\emptyset)}$ and similar for the other non-generic cases. In particular, $\chi_{(\emptyset,[\lambda:\mu])}^2 = \chi_{(\{a,b,c\},\emptyset)}$ for \newline$[\lambda:\mu] \neq [1:1], [1:-1], [1:0], [0:1], [1:i], [1:-i]$.\\

To calculate the product $\chi_{(\emptyset,\ast)}\chi_{(\emptyset,[\lambda:\mu])}$ we may present the glider associated to the first factor by
$$U \oplus U \supset \mathbb{C}(e_1 + e_2) \supset 0 \supset \ldots$$
Hence the tensor product with $U \supset \mathbb{C}(\lambda e_1 + \mu e_2) \supset 0 \supset \ldots$ becomes
$$(U \ot U) \oplus (U \ot U) \supset \mathbb{C}(\underbrace{(\lambda e_1 \ot e_1 + \mu e_1 \ot e_2)}_{v_1} + \underbrace{(\lambda e_2 \ot e_1 + \mu e_2 \ot e_2)}_{v_2}) \supset 0 \supset \ldots$$
We denote the coefficients of the decomposition of $v_i$ with regard to \eqref{decoUU} by $a_i,b_i,c_i,d_i,~i=1,2$. The vector $v_1$ yields the equation 
$$\begin{pmatrix} \lambda & \mu & 0 & 0 \end{pmatrix} = \begin{pmatrix} c_1+d_1 & a_1+b_1 & b_1-a_1 & c_1-d_1 \end{pmatrix},$$
from which we obtain that $a_1b_1c_1d_1 \neq 0$ unless $\lambda \mu =0$. If, say, $\lambda =0$, then $c_1=d_1=0$. In this case, the decomposition of the $v_2$ yields that $c_2d_2 \neq 0$. So we arrive at
$$\chi_{(\emptyset,\ast)}\chi_{(\emptyset,[\lambda:\mu])} = \chi_{(V_4,\emptyset)}, \quad \forall~ [\lambda:\mu] \in \mathbb{P}^1.$$
A similar reasoning entails that
$$\chi_{(\emptyset,\ast)}^2 = \chi_{(V_4,\emptyset)}.$$
Next, to calculate the products of the form $\chi_{(A,\emptyset)}\chi_{(\emptyset,[\lambda:\mu])}$, we need to write explicit isomorphisms $T_r \ot U \cong U$ for $r = 1,i,j,k$. These are given by
$$\begin{array}{c|c}
U \mapright{\varphi_1} T_1 \ot U & U \mapright{\varphi_i} T_i \ot U \\
e_1 \mapsto t_1 \ot e_1 & e_1 \mapsto t_i \ot e_2\\
e_2 \mapsto t_1 \ot e_2 & e_2 \mapsto t_i \ot e_1\\
\hline
&\\
U \mapright{\varphi_j} T_j \ot U & U \mapright{\varphi_k} T_k \ot U \\
e_1 \mapsto -t_j \ot e_1 & e_1 \mapsto -t_k \ot e_2\\
e_2 \mapsto t_j \ot e_2 & e_2 \mapsto t_k \ot e_1
\end{array} \Rightarrow \left\{\begin{array}{l}
\chi_{(\{1\},\emptyset)}\chi_{(\emptyset,[\lambda:\mu])} = \chi_{(\emptyset,[\lambda:\mu])}\\
\chi_{(\{a\},\emptyset)}\chi_{(\emptyset,[\lambda:\mu])} = \chi_{(\emptyset,[\mu:\lambda])}\\
\chi_{(\{b\},\emptyset)}\chi_{(\emptyset,[\lambda:\mu])} = \chi_{(\emptyset,[-\lambda:\mu])}\\
\chi_{(\{c\},\emptyset)}\chi_{(\emptyset,[\lambda:\mu])} = \chi_{(\emptyset,[-\mu:\lambda])}
\end{array}\right.$$
For subsets $A \subset V_4$ of two elements, we generically have
\begin{equation}\label{AP}
\chi_{(A,\emptyset)}\chi_{(\emptyset,[\lambda:\mu])} = \chi_{(\emptyset,\ast)},
\end{equation}
but there are a few special cases. For example, for $A = \{1,b\}$ we have that $\chi_{(\{1,b\},\emptyset)}\chi_{(\emptyset,[0:1])} = \chi_{(\emptyset,[0:1])}$, because $[0:1] = [\lambda:\mu] = [-\lambda: \mu] = [1:0]$. We enlist all these special cases in the following table
$$\begin{array}{c|l|l}
A & {\rm non-generic~points~} [\lambda:\mu] &  \chi_{(A,\emptyset)}\chi_{(\emptyset,[\lambda:\mu])}\\
\hline 
\{1,a\} & [1:1], [1:-1] & \chi_{(\emptyset,[1:1])}, \chi_{(\emptyset,[1:-1])}\\
\{b,c\} & [1:1], [1:-1] & \chi_{(\emptyset,[-1:1])}, \chi_{(\emptyset,[1:1])}\\
\{1,b\} & [0:1], [1:0] & \chi_{(\emptyset,[0:1])}, \chi_{(\emptyset,[1:0])}\\
\{a,c\} & [0:1], [1:0] & \chi_{(\emptyset,[1:0])}, \chi_{(\emptyset,[0:1])}\\
\{1,c\} & [1:i], [1:-i] & \chi_{(\emptyset,[1:i])}, \chi_{(\emptyset,[1:-i])}\\
\{a,b\} & [1:i], [1:-i] & \chi_{(\emptyset,[1:-i])}, \chi_{(\emptyset,[1:i])}
\end{array}$$
For $|A| \geq 3$, formula \eqref{AP} holds for all $[\lambda:\mu] \in \mathbb{P}^1$. Finally, it easily follows that 
$$\chi_{(A,\emptyset)}\chi_{(\emptyset,\ast)} = \chi_{(\emptyset,\ast)}, \quad \forall A \in \mc{P}(V_4).$$

All the other products in $R(\wt{Q_8})$ can now be deduced easily. For example
$$\chi_{(\{a\},[1:i])}\chi_{(\{b\},[1:i])} = \chi_{(\{a,b,c\},[1:-i])},$$
and
$$\chi_{(\{a,b\},\ast)}\chi_{(\emptyset,[\lambda:\mu])} = \chi_{(V_4,\emptyset)}.$$
We also include the exponents of $\chi_{(\emptyset, [1:1])}$
$$\begin{array}{c|c}
n & \chi_{(\emptyset, [1:1])}^n\\
\hline
1 & \chi_{(\emptyset, [1:1])}\\
2 & \chi_{(\{b,c\}, \emptyset)}\\
3 & \chi_{(\emptyset, [-1:1])}\\
4 & \chi_{(\{1,a\}, \emptyset)}\\
5 & \chi_{(\emptyset, [1:1])} \\
\vdots & \vdots
\end{array}$$

Now that we know how to multiply in $R(\wt{Q}_8)$, whence also in $\mathbb{Q}(\wt{Q}_8)$, we wonder what the Jacobson radical $J = J(\mathbb{Q}(\wt{Q}_8))$ looks like. For finite abelian groups $G$, we observed that $\chi_G$ is a 0-element in the semigroup. For $Q_8$ the element $\chi_{(V_4,\ast)}$ is easily seen to be a 0-element.
\begin{lemma}\lemlabel{exponents}
For any non-empty subset $A \subset V_4$ and $[\lambda:\mu] \in \mathbb{P}^1$ there exists an $n > 0$ such that $\chi^n_{(A,[\lambda:\mu])} = \chi_{(V_4,\ast)}$.
\end{lemma}
\begin{proof}
From \eqref{AP} it follows that $n = 2$ suffices for $|A| \geq 3$. Generically we have that $\chi_{(\emptyset,[\lambda:\mu])}^2 = \chi_{(\{a,b,c\},\emptyset)}$, hence generically the result also follows for $|A| = 2$. For the non-generic points one has to check all the cases. For example
$$\chi_{(\{b,c\},[1:1])}^2 = \chi_{(V_4, [1:-1])} \Rightarrow n = 3 {\rm~satisfies.}$$
Finally if $|A| = 1$, then $\chi_{(A,[\lambda:\mu])}^2 = \chi_{(A', [\lambda':\mu'])}$ with $|A'| \geq 2$ and the result again follows.
\end{proof}
\begin{corollary}\corlabel{jacext}
The inclusion of $\mathbb{Q}$-algebras
$$\mathbb{Q}(\wt{V}_4) \hookmapright{} \mathbb{Q}(\wt{Q}_8),~ \chi_A \mapsto \chi_{(A,\emptyset)}$$
is an integral extension.
\end{corollary}
\begin{proof}
The 0-element $\chi_{(V_4,\ast)}$ is idempotent, hence integral. The elements $\chi_{(A,[\lambda:\mu])}$ for $A$ non-empty are integral by \lemref{exponents}. The elements $\chi_{(A,\ast)}$ are integral since $\chi_{(A,\ast)}^2 = \chi_{(V_4,\ast)}$. Finally, the elements $\chi_{(\emptyset,[\lambda:\mu])}$ are integral since their squares sit in $\mathbb{Q}(\wt{V}_4)$. 
\end{proof}
As another corollary of \lemref{exponents} we know that $\chi_{(A,[\lambda:\mu])} = \chi_{(V_4,\ast)}$ modulo the Jacobson radical $J$ for all $A$ non-empty and $[\lambda:\mu] \in \mathbb{P}^1$. For $A = \emptyset$ and $[\lambda:\mu]$ a generic point one calculates that $(\chi_{(\emptyset,[\lambda:\mu])} - \chi_{(\emptyset,\ast)})^3 = 0$. Denote by $I'$ the ideal generated by 

$$\chi_{(\emptyset,[\lambda:\mu])} - \chi_{(\emptyset,\ast)}, ~ [\lambda:\mu] \in \mathbb{P}^1 {\rm~generic},$$
and 
$$\chi_{(A,[\lambda:\mu])} - \chi_{(V_4,\ast)},~ A \subset V_4 {\rm~non-empty},  [\lambda:\mu] \in \mathbb{P}^1.$$
By \corref{jacext} $J(\mathbb{Q}(\wt{V}_4)) = J \cap \mathbb{Q}(\wt{V}_4)$, hence we have an inclusion
$$\iota: \mathbb{Q}(\wt{V}_4)/J(\mathbb{Q}(\wt{V}_4)) \hookmapright{} \mathbb{Q}(\wt{Q}_4)/J.$$
From \theref{Ja} we know how the left hand side decomposes but there is a small remark here. In our calculation in the abelian case, we worked in the contracted semigroup algebra. The element $\chi_{(V_4,\emptyset)}$ is however no longer a 0-element in the bigger algebra, hence we have the isomorphism
$$\varphi: \bigoplus_{H \leq V_4} \mathbb{Q}(V_4/H) \mapright{\cong} \mathbb{Q}(\wt{V}_4)/J(\mathbb{Q}(\wt{V}_4)),$$
where the element $e_{V_4}$ is mapped to $\chi_{V_4}$. 
\begin{theorem}
The Jacobson radical $J$ equals $I' + J(\mathbb{Q}(\wt{V}_4))$ and we have a ring isomorphism
$$\mathbb{Q}(\wt{Q}_8)/J \cong \mathbb{Q}(V_4) \oplus \mathbb{Q}(\mathbb{Z}_4)^{\oplus 3} \oplus \mathbb{Q}^{\oplus 2} \oplus \mathbb{Q}.$$
\end{theorem}
\begin{proof}
By definition of $I'$ we have that 
$$\dim_\mathbb{Q}(\frac{\mathbb{Q}(\wt{Q}_8)}{I' + J(\mathbb{Q}(\wt{V}_4))}) = \dim_\mathbb{Q}(\frac{\mathbb{Q}(\wt{V}_4)}{J(\mathbb{Q}(\wt{V}_4))}) + 8.$$
Further, recall from \theref{Ja} the decomposition of $\mathbb{Q}(\wt{V}_4)/J(\mathbb{Q}(\wt{V}_4))$ with resp. unit elements
\begin{equation}\label{tabeldeco}
\begin{array}{l|l}
\mathbb{Q}(V_4/\{1\}) & (1 - \chi_{\{1,a\}})(1 - \chi_{\{1,b\}})(1 - \chi_{\{1,c\}})\\
\mathbb{Q}(V_4/\{1,a\}) & (\chi_{\{1,a\}} - \chi_{V_4})\\
\mathbb{Q}(V_4/\{1,b\}) & (\chi_{\{1,b\}} - \chi_{V_4})\\
\mathbb{Q}(V_4/\{1,c\}) & (\chi_{\{1,c\}} - \chi_{V_4})\\
\mathbb{Q}(V_4/V_4) & \chi_{V_4}
\end{array}
\end{equation}
Inside $\mathbb{Q}(\wt{Q}_8)$ we have that
$$\chi_{(V_4,\emptyset)} = (\frac{1}{2}\chi_{(V_4,\emptyset)} + \frac{1}{2}\chi_{(\emptyset,\ast)}) +  (\frac{1}{2}\chi_{(V_4,\emptyset)} - \frac{1}{2}\chi_{(\emptyset,\ast)}),$$
and the elements $\frac{1}{2}\chi_{(V_4,\emptyset)} \pm \frac{1}{2}\chi_{(\emptyset,\ast)}$ are orthogonal idempotents, which are also orthogonal to the first four idempotents of \eqref{tabeldeco}. Next, we consider the inclusion $\iota(\varphi(\mathbb{Q}(V_4/\{1,a\}))$. The element $\iota\varphi(e_{b\{1,a\}}) = \iota(\chi_{\{b,c\}} - \chi_{V_4})$ is the square of 
$$\chi_{(\emptyset,[1:1])} - \chi_{(\emptyset,\ast)},$$ 
and we see that $ \iota(\varphi(\mathbb{Q}(V_4/\{1,a\}))$ is the subgroup algebra $\mathbb{Q}(\mathbb{Z}_2)$ of 
$$\mathbb{Q}(\langle\chi_{(\emptyset,[1:1])} - \chi_{(\emptyset,\ast)}\rangle) \cong \mathbb{Q}(\mathbb{Z}_4).$$
The inclusions $\iota(\varphi(\mathbb{Q}(V_4/\{1,b\}))$ and $\iota(\varphi(\mathbb{Q}(V_4/\{1,c\}))$ behave analogously and they use the non-generic points $[1:0], [0:1]$ and $[1:i],[1:-i]$ respectively. The isomorphism now follows since the elements
$$\begin{array}{ll}
(1 - \chi_{(\{1,a\},\emptyset)})(1 - \chi_{(\{1,b\},\emptyset)})(1 - \chi_{(\{1,c\},\emptyset)}) & \frac{1}{2}\chi_{(V_4,\emptyset)} \pm \frac{1}{2}\chi_{(\emptyset,\ast)}\\
 \chi_{(\{1,a\},\emptyset} - \chi_{(V_4,\emptyset)} & \chi_{(\{1,b\},\emptyset)} - \chi_{(V_4,\emptyset)} \\
 \chi_{(\{1,c\},\emptyset)} - \chi_{(V_4,\emptyset)} & \chi_{(V_4,\ast)}
 \end{array}$$
are all orthogonal idempotents.
\end{proof}

\section*{Acknowledgement}

The first author expresses his gratitude to Geoffrey Janssens and Eric Jespers for fruitful discussions on the calculation of the Jacobson radical and the primitive central idempotents of semigroup algebras appearing in the theory of glider representations.

\end{document}